\documentclass[12pt]{amsart}
\usepackage[margin=1.0in]{geometry}

\usepackage[T1]{fontenc}
\usepackage[utf8]{inputenc}
\usepackage{lmodern}
\usepackage{hyperref, amsmath, amssymb, mathtools}
\usepackage{verbatim}
\usepackage{amsthm}
\usepackage{amsfonts}
\usepackage{amsmath}
\usepackage{mathrsfs}  
\usepackage{comment}
\usepackage{chngcntr}
\usepackage{blkarray}
\usepackage{tikz-cd}
\usepackage{ctable}

\newtheorem{thm}{Theorem}
\newtheorem*{thm*}{Theorem}
\newtheorem{lemma}[thm]{Lemma}
\newtheorem{corollary}[thm]{Corollary}

\newtheorem{prop}[thm]{Proposition}
\newtheorem*{prop*}{Proposition}

\theoremstyle{definition}

\theoremstyle{remark}
\newtheorem{remark}[thm]{Remark}

\newtheorem*{example*}{Example}
\numberwithin{thm}{section}
\numberwithin{equation}{section}

\usepackage{xcolor}

\newcommand{\R}{\mathbb{R}}
\newcommand{\Q}{\mathbb{Q}}
\newcommand{\Z}{\mathbb{Z}}
\newcommand{\F}{\mathbb{F}}

\newcommand{\C}{\mathbb{C}}
\newcommand{\A}{\mathbb{A}}
\newcommand{\GL}{\text{GL}}

\newcommand{\calO}{\mathcal{O}}

\newcommand{\bx}{\mathbf{x}}

\newcommand{\fp}{\mathfrak{p}}

\newcommand{\fa}{\mathfrak{a}}
\newcommand{\frakP}{\mathfrak{P}}

\newcommand{\Ind}{\text{Ind}}
\newcommand{\Tr}{\text{Tr}}

\newcommand{\Frob}{\textrm{Frob}}

\newcommand{\ord}{\text{ord}}

\newcommand{\Gal}{\text{Gal}}
\newcommand{\sym}{\text{sym}}

\mathtoolsset{showonlyrefs}

\title{Sums of Hecke eigenvalues along polynomial sequences and base change for $\GL(2)$}
\author[K. Woo]{Katharine Woo}
\address{Department of Mathematics, Stanford University, Stanford, CA 94305}
\email{katywoo@stanford.edu}
\date{\today}

\begin{document}

\maketitle

\begin{abstract}

    We study sums of absolute values of Hecke eigenvalues of $\textrm{GL}(2)$ representations that are tempered at all finite places. We show that these sums exhibit logarithmic savings over the trivial bound if and only if the representation is cuspidal. Further, we connect the problem of studying the sums of Hecke eigenvalues along polynomial values to the base change problem for $\textrm{GL}(2).$
\end{abstract}

\section{Introduction}\label{app: Hecke eigenvalues}

Let $\pi$ be a cuspidal automorphic representation of $\GL_2(\A_\Q)$ that satisfies the Ramanujan-Petersson conjecture; let $\lambda_\pi$ denote the normalized Hecke eigenvalues of $\pi$. Fix an irreducible nonconstant solvable polynomial $P(x)\in \Z[x]$. In this paper, we focus on the following correlation sum:
\begin{equation}\label{eq: main sum}
    \sum_{n\leq X} |\lambda_\pi(|P(n)|)|.
\end{equation}
In particular, we connect the above sum to the base change of $\pi$ to the splitting field of $P(x)$, denoted as $K$. In the following theorem, we give a sufficient condition for \eqref{eq: main sum} to exhibit extra savings over the trivial bound obtained using the bound $|\lambda_\pi(p)|\leq 2$.
\begin{thm}\label{thm: poly}
Let $\pi$ be a cuspidal automorphic representation of $\GL_2(\A_\Q)$ that satisfies the Ramanujan-Petersson conjecture. Let $P(x)\in \Z[x]$ be irreducible and solvable, and let $K$ denote the splitting field of $P(x)$. Assume that the \textbf{base change} $\pi_K$ to $\GL_2(\A_K)$ is \textbf{cuspidal}, then the following bound holds: 
$$\sum_{n\leq X} |\lambda_\pi(|P(n)|)| \ll_{\pi,P} X \log(X)^{-0.066}.$$
\end{thm}
Observe that Theorem \ref{thm: poly} applies to those $\pi$'s corresponding to holomorphic cusp forms (\cite{DeligneReduction,DeligneSerre}). One may expect that there should be further, even power-saving, cancellation on the sum 
\begin{equation}
    \sum_{n\leq X} \lambda_\pi(|P(n)|).
\end{equation}
Indeed, when $P(x)$ is a linear polynomial, this stronger bound exists. Blomer \cite{BlomerSumsOverQuadratics} analyzes sums of the above form when $P(x)$ is quadratic and $\pi$ corresponds to a holomorphic cusp form of weight $k\geq 4$, and determines when the sum exhibits power-saving cancellation; these sums were further studied in \cite{Templier,TemplierTsimerman} by Templier and Tsimerman. However, when $P(x)$ is of higher degree, it is not clear how to successfully approach these sums (as discussed by Sarnak in \cite{Sarnak_2001}). When $\pi$ is no longer cuspidal, asymptotics are known for certain variants of this problem -- Greaves \cite{Greaves} and Daniel \cite{Daniel} derive asymptotics for this expression when $\lambda_\pi$ is replaced by the divisor function, which is the Fourier coefficient of an Eisenstein series, along binary cubic and binary quartic forms respectively; most recently, Leung and Pandey proved an asymptotic with a power-saving error term in the case the divisor function along binary quartic forms \cite{LeungPandey}.

By instead considering \eqref{eq: main sum}, which has absolute values around the Hecke eigenvalues, we gain flexibility in the polynomial arguments that we can study; the caveat of this approach is that one should expect a limit of a logarithmic power saving. Consider, for example, when $P(x) = x$: if $\pi$ satisfies the Sato-Tate distribution (see \cite{CHT,Taylor}), then the following asymptotic for the absolute values of the Hecke eigenvalues holds (see the concluding remarks of \cite{EMS} by Elliot, Moreno, and Shahidi):
\begin{equation}
    \sum_{n\leq X} |\lambda_\pi(n)| \sim \frac{cX}{\log(X)^{1-\frac{8}{3\pi}}}.
\end{equation}

There have been many successful upper bounds for cases of \eqref{eq: main sum}. Elliot, Moreno, and Shahidi \cite{EMS} considered the case of the Ramanujan $\tau$-function and their method forms the basis for our proof. Holowinsky \cite{HolowinskyShiftedConvolution} studied the case of \eqref{eq: main sum} for the polynomials $x(x+\ell)$; Nelson \cite{Nelson} later studies the case when the polynomial is a irreducible quadratic. For general polynomial entries, Chiriac and Yang \cite{ChiriacYang} considered when $\pi$ is defined by a non-CM holomorphic newform with trivial character along arbitrary polynomials. 

We also briefly remark that sums of the form \eqref{eq: main sum} have numerous applications in number theory. They played a critical role in the resolution of the holomorphic analogue of QUE by Holowinsky and Soundararajan \cite{HolowinskyQUE,HolowinskySoundQUE}, i.e. mass equidistribution for Hecke eigenforms; notably the bounds achieved by Holowinsky depended logarithmically on the coefficients of the polynomial arguments $x(x+\ell)$ and hence summing over $\ell$ is well-controlled. These sums appear in the recent proof of the mixing conjecture conditional on GRH by Blomer, Brumley, and Khayutin \cite{BlomerBrumleyKhayutin}. The author began studying these correlation sums while working on Manin's conjecture for Ch\^atelet surfaces \cite{MyManinChatelet}, where nontrivial upper bounds on these sums are applied in the indefinite case.

\begin{remark}
    In Theorem \ref{thm: poly}, we focus on the case where $P(x)$ is an irreducible polynomial. However, for applications of this bound, one may desire for $P(x)$ to be a squarefree polynomial. The sieve used (Theorem \ref{thm: sieve nair}) has been generalized to squarefree polynomials and from these generalizations the following can be deduced: 
    
    Let $P(x) = P_1(x)\hdots P_r(x)$ be a squarefree polynomial and $P_i(x)$ irreducible factors. Let $K_i$ be the splitting field of $P_i(x)$. Assume that the base change $\pi_{K_i}$ is cuspidal for all $i\in \{1,\hdots,r\}$. Then there exists a positive $c_r>0$ such that
    \begin{equation}
        \sum_{n\leq X}|\lambda_\pi(|P(n)|)| \ll_{\pi,P} X\log(X)^{-c_r}.
    \end{equation}

\end{remark}

In \S\ref{sec: poly}, we establish the connection between \eqref{eq: main sum} and the base change of $\pi$ to the splitting field of $P(x)$. In doing so, the problem reduces to the linear case over number fields. After introducing a new argument exploiting the structure of dihedral representations, we establish the following classification for the $L^1$-norm:
\begin{thm}\label{thm: linear}
    Let $K$ be a number field and $\pi$ an automorphic representation of $\GL_2(\A_K)$ that is tempered at all finite places of $K$. Then there exists a positive constant $c_\pi>0$ such that 
    $$\sum_{\substack{\fa\subset \calO_K \\ N(\fa)\leq X}} |\lambda_\pi(\fa)| \ll_\pi X \log(X)^{-c_\pi}$$
    if and only if $\pi$ is cuspidal. Moreover, if $\pi$ is cuspidal, then there is a universal lower bound $c_\pi>0.066$. 
\end{thm}

The work of Rankin \cite{Rankin} and Selberg \cite{Selberg} for the $L^2$-norm of cuspidal representations tells us that there is also a power of $\log(X)$ difference depending on the cuspidality of $\pi$ in that story -- in particular, that the asymptotic
$$\sum_{n\leq X} |\lambda_\pi(n)|^2 = c_\pi X + O(X^{1-2/5})$$
holds if and only if $\pi$ is a cuspidal automorphic representations of $\GL_2(\A_\Q)$. If $\pi$ is not cuspidal, then the $L^2$-norm is of size $\asymp X\log(X)^{m}$ for $m\geq 1$. Observe that by Cauchy-Schwarz, this implies that the $L^1$-norm, i.e. the sum studied in Theorem \ref{thm: linear}, is bounded by $O_\pi(X)$ when $\pi$ is cuspidal. We will use analysis on the higher symmetric powers to establish our upper bound. 
\begin{remark}
    We remark that we assume that $\pi$ is tempered at all finite places in order to apply our sieves (Theorem \ref{thm: sieve nair} and Theorem \ref{thm: lower bound sieve}). In \cite{Yang}, Yang studies upper and lower bounds for this $L^1$-norm when $\pi$ is a cuspidal representation of $\GL_2(\A_\Q)$ that is nonmonomial and unitary, without the assumption that $\pi$ satisfies the Ramanujan-Petersson conjecture. For the nonlinear case of Theorem \ref{thm: poly}, we currently need this temperedness assumption to apply an appropriate sieve.   
\end{remark}
\medskip

Finally, we hope to understand \eqref{eq: main sum} when $\pi_K$ is not a cuspidal representation. The work on the base change of $\GL_2$ representations by Langlands \cite{LanglandsBaseChange} classifies exactly when this phenomenon occurs for $P(x)$ a solvable polynomial; we use this to establish lower bounds for \eqref{eq: main sum}. 
\begin{thm}\label{thm: noncuspidal}
    Let $\pi$ be a cuspidal automorphic representation of $\GL_2(\A_\Q)$ that satisfies the Ramanujan-Petersson conjecture. Let $P(x)\in \Z[x]$ be irreducible and $K$ be the splitting field of $P(x)$ be solvable. Assume that the base change $\pi_K$ is not cuspidal; hence, it has the form $$\pi = \Ind_{W_F}^{W_\Q} \theta$$ for some quadratic subfield $F\subset K$ and unitary Hecke character $\theta$ of $F$. 
    If $\theta$ has finite order, let $L_\theta$ denote the ray class field of $\theta$. 
    
    \begin{enumerate}
        \item If $\theta$ is of infinite order or if $\Gal(L_\theta/L_\theta\cap K) \not\subset \ker(\theta^2)$, then
    $$\sum_{n\leq X} |\lambda_\pi(|P(n)|)| \ll_{\pi,P} X\log(X)^{1/2}.$$
    \item Otherwise, the following lower bound holds: $$\sum_{n\leq X} |\lambda_\pi(|P(n)|)| \geq  \sum_{\substack{n\leq X\\ P(n) \text{ squarefree}}} \sum_{\substack{\fa\subset \calO_F \\ N(\fa) = |P(n)|}}1.$$
    \end{enumerate}
    
\end{thm}
\begin{remark}
    A necessary requirement for $\pi_K$ to be noncuspidal is that $F\subset K$; this correspondingly increases the size of our ``trivial'' bound -- the one induced by the trivial representation. Specifically, let $\Psi = \Ind_{W_F}^{W_\Q} 1$. If $F\subset K$, then we have that 
    $$\sum_{n\leq X} |\lambda_\Psi(|P(n)|)| \ll _{\Psi,P}X\log(X).$$
    In comparison, if $F\not\subset K$, then the trivial bound gives: 
    $$\sum_{n\leq X}|\lambda_\Psi(|P(n)|)| \ll_{\Psi,P} X.$$ When $F\subset K$, this extra power of $\log(X)$ is a consequence of the Dedekind zeta function $\zeta_K(s)$ twisted by the quadratic character generating $F$ having a pole at $s=1$ of order one. Hence Theorem \ref{thm: noncuspidal} (1) can also be viewed as a polylogarithmic savings over the trivial bound.
\end{remark}

Finally, we remark on a natural generalization of this problem -- when $P(x)$ is replaced by an absolutely irreducible \textit{multivariate} polynomial:
\begin{equation}
    \sum_{x_i\leq X} |\lambda_\pi(|P(x_1,\hdots,x_n|)|.
\end{equation}
First, if $P(x_1,\ldots,x_n)$ is a homogeneous binary form, then we achieve the same theorems as the single variable case (see the remarks after Theorem \ref{thm: sieve nair}). In particular, if $K$ is the splitting field of $P(\bx)$ and if the base change $\pi_K$ is cuspidal, then 
$$\sum_{x,y\leq X} |\lambda_\pi(|P(x,y)|)| \ll X^2 \log(X)^{-0.066}.$$

For $n\geq 3$ or for nonhomogeneous polynomials with $n=2$, after an appropriate sieve for nonnegative multiplicative functions along polynomial values (as derived by Chiriac and Yang in \cite[Proposition 2.2]{ChiriacYang}), this problem reduces to questions about the Sato-Tate distribution of $\pi$ and the number of local solutions of $P(\bx) \bmod p$. Under the assumption of a wide range of conjectures (including those connecting the Hasse-Weil $L$-function of the variety $\{P(x_1,\hdots,x_n)=0\}$ to automorphic representations), one can achieve conditional polylogarithmic savings over the trivial bound by using second and fourth moment calculations. 

\subsection{Notation}
Let us first set some notation: 
\begin{itemize}
    \item $p$, $\fp$, or $\frakP$ will denote primes or prime ideals,
    \item sums over $p$, $\fp$, or $\frakP$ are sums over primes or prime ideals,
    \item $\varrho(p) = \#\{x\in \F_p: P(x) = 0\}$ will denote the local count of the polynomial $P(x)$,
    \item $\pi$, $\Xi$ will denote $\GL_2$ representations, 
    \item $\psi$, $\xi$, $\theta$ will denote $\GL_1$ representations, i.e. Hecke characters,
    \item Let $\Pi(G)$ denote the irreducible admissible automorphic representations of $G$. The base change map sends $$\Pi(\GL_2(\A_\Q))\rightarrow \Pi(\GL_2(\A_K))$$ and is analyzed for solvable extensions by Langlands in \cite{LanglandsBaseChange}.
\end{itemize}

We end this section by outlining the rest of the paper. In \S\ref{sec: sieve}, we establish the necessary background on sieve theory. In \S\ref{sec: linear}, we establish Theorem \ref{thm: linear}. In \S\ref{sec: poly}, we first prove Theorem \ref{thm: poly} by reducing it to the setting of Theorem \ref{thm: linear}; second, we consider non-cuspidal base change and establish Theorem \ref{thm: noncuspidal}.

\section*{Acknowledgements}
The author would like to thank her advisor Peter Sarnak for his guidance and support. The author would further like to thank Peter Humphries, Sam Mundy, Paul Nelson, Naomi Sweeting, and Liyang Yang for helpful conversations and comments.

This material is based upon work supported by the National Science Foundation Graduate Research Fellowship Program under Grant No. DGE-2039656. Any opinions, findings, and conclusions or recommendations expressed in this material are those of the author(s) and do not necessarily reflect the views of the National Science Foundation.

\section{Sieve methods}\label{sec: sieve}

\subsection{Upper bound sieves}
We begin with a classical upper bound. Let $f(n)$ be a nonnegative multiplicative function supported on squarefree numbers. Then if $f(n)$ satisfies the following condition for the sum over primes:
\begin{equation}
    \sum_{w\leq p<X} \frac{f(p)\log(p)}{p} \leq a \log(X/w) + b,
\end{equation}
for all $w$ with $2\leq w<X$ and $a\geq 0,b\geq 1$ constants, then
one can derive that\footnote{See \cite[Theorem A.1]{OperaDeCribro} for further details} 
\begin{equation}
    \sum_{n\leq X} f(n) \ll X \prod_{p\leq X} \left(1+\frac{f(p)-1}{p}\right).
\end{equation}

Now, let $P(x)\in \Z[x]$ be a polynomial -- we would like a similar upper bound along values of $P$. The following sieve was deduced by Nair \cite{Nair} for nonnegative multiplicative functions along polynomial values. It was later generalized by de la Bret\'eche and Browning \cite{delaBrowningSieve} and de la Bret\'eche and Tenenbaum \cite{delaTenenbaumSieve} for binary forms; by Browning and Sofos \cite{BrowningSofos} for sums over prime ideals. 
\begin{thm}[Nair, \cite{Nair}]\label{thm: sieve nair}
    Let $f(n)$ be a nonnegative multiplicative function such that there exist constants $A,B\geq 0$ such that for any natural number $v$, $f(p^v)\leq A^v$, and in general, $f(n)\leq Bn^\epsilon$. Then the following bound holds for any irreducible polynomial $P(x)\in \Z[x]$: 
    \begin{equation}
        \sum_{n\leq X} f(|P(n)|) \ll_{P,A,B} X \exp\left(\sum_{p\leq X} \frac{\varrho(p)(f(p)-1)}{p}\right),
    \end{equation}
    where $\varrho(p) = \#\{a\bmod p: P(a)=0\}.$ 
\end{thm}
\begin{remark}
Later, Nair and Tenenbaum \cite{NairTenenbaum} generalized the above result for multiplicative functions in several variables. A consequence of their work is that if $P_1(x),\hdots,P_k(x)\in \Z[x]$ are distinct irreducible polynomials and $f_1,\hdots,f_k$ are nonnegative multiplicative functions that are bounded on prime values, then 
\begin{equation}\label{eq: sieve many entries}
    \sum_{n\leq X} f_1(|P_1(n)|)\hdots f_k(|P_k(n)|) \ll_{P_1,\hdots, P_k,f_1,\hdots, f_k} X \exp\left(\sum_{i=1}^{k} \sum_{p_i\leq X} \frac{\varrho_{P_i}(p_i)(f_i(p_i)-1)}{p_i}\right),
\end{equation}
where $\varrho_{P_i}(p) = \#\{a\bmod p: P_i(a) = 0\}.$ A consequence of the work of de la Bret\'eche and Tenenbaum \cite{delaTenenbaumSieve} for binary forms is that a similar bound holds for $g_1,\hdots,g_k$ nonnegative multiplicative functions bounded on prime values and $F_1,\hdots, F_k$ distinct irreducible binary forms: 
\begin{equation}\label{eq: sieve binary form}
\sum_{x,y\leq X} \prod_{i=1}^k g_i(|F_i(x,y)|) \ll_{g_1,\hdots, g_k, F_1,\hdots, F_k} X^2 \exp\left(\sum_{i=1}^k \sum_{p_i\leq X} \frac{\varrho_{F_i}(p_i)(g_i(p_i)-1)}{p_i^2}\right),\end{equation}
where $\varrho_{F_i}(p) = \#\{(a,b) \bmod p: F(a,b) = 0\}.$ For this paper, we will stick however to the case of one irreducible polynomial, $P(\bx)$, and the particular nonnegative multiplicative functions $|\lambda_\pi(n)|$.
\end{remark}

\subsection{Lower bound sieves}
In Theorem \ref{thm: linear}, we also want to consider lower bounds. To do so, we use a corresponding lower bound sieve. 
\begin{thm}[{\cite[Theorem A.4]{OperaDeCribro}}]\label{thm: lower bound sieve}
Let $f(n)$ be a multiplicative nonnegative function and satisfy that:
\begin{equation}
    \sum_{y<p\leq x} \frac{f(p)\log(p)}{p} \leq a \log(x/y) +b,
\end{equation}
\begin{equation}
    \sum_{p\leq y}f(p)\log(p) \gg y.
\end{equation}
Then $f(n)$ satisfies that 
\begin{equation}
    \sum_{n\leq X}f(n) \gg  X\prod_{p\leq X} \left(1+\frac{f(p)-1}{p}\right).
\end{equation}
    
\end{thm}
\begin{remark}
    If $f(n)$ is bounded on prime values, then the first condition is always satisfied. 
\end{remark}

\begin{corollary}\label{cor: lower bd number fields}
Let $K$ be a number field and $g:I_K\rightarrow \R$ be a nonnegative multiplicative function on the ideals of $K$ that is bounded on prime ideals. If $$\sum_{\substack{\fp\subset \calO_K\\ N(\fp)\leq y}} g(\fp) \log(N(\fp))\gg y,$$
then the following bound is satisfied:
$$\sum_{\substack{\fa\subset \calO_K \\ N(\fa) \leq X}}g(\fa) \gg_K X \prod_{N(\fp)\leq X} \left(1+\frac{g(\fp)-1}{N(\fp)}\right).$$
\end{corollary}
\begin{proof}
    Apply Theorem \ref{thm: lower bound sieve} with the nonnegative multiplicative function $$f(n) = \sum_{\substack{\fa\subset \calO_K \\ N(\fa)=n}} g(\fa).$$
\end{proof}
\begin{remark}[Non-existence of a lower bound sieve for polynomial values]
    One might expect, given Theorem \ref{thm: lower bound sieve}, that a similar bound would hold for the sum: $$\sum_{n\leq X} f(|P(n)|).$$
    However, this can not happen; for instance, take $f(p) = \mathbf{1}_{p\equiv 3 \bmod 4}$ and $P(x) = x^2+1$. Then it is clear that $$\sum_{n\leq X}f(|P(n)|)=0.$$
    However, if we consider what the analogous lower bound to Theorem \ref{thm: lower bound sieve} would produce,
    $$X\prod_{p\leq X} \left(1+\frac{\varrho(p)(f(p)-1)}{p}\right) = X \prod_{\substack{p\leq X \\ p\equiv 1 \bmod 4}}\left(1-\frac{2}{p}\right)\gg X\log(X)^{-1},$$
    this causes immediate issues.  
\end{remark}

\section{Linear polynomials}\label{sec: linear}
In this section, we prove Theorem \ref{thm: linear}. Recall that $K$ is a number field and $\pi$ is an automorphic representation of $\GL_2(\A_K)$ that is tempered at all finite places of $K$. Consequently, we know that the normalized Hecke eigenvalues satisfy that 
$$|\lambda_\pi(\fp)| \leq 2,$$
for all prime ideals $\fp\subset \calO_K$. 

To prove the theorem, we split the study into different cases. In the first two subsection, we prove that the upper bound holds whenever $\pi$ is cuspidal. In the final subsection (in \S\ref{subsec: linear lower}), we will show that when $\pi$ is not cuspidal, a sufficient lower bound holds. \medskip

For the upper bound, we first apply Theorem \ref{thm: sieve nair}, as the assumptions are satisfied: 
\begin{equation}
    \sum_{N(\fa)\leq X} |\lambda_\pi(\fa)| \ll_K X\exp\left(\sum_{N(\fp)\leq X} \frac{|\lambda_\pi(\fp)|-1}{N(\fp)}\right).
\end{equation}
In order to study the sum
$$S_\pi(X) := \sum_{N(\fp)\leq X} \frac{|\lambda_\pi(\fp)|-1}{N(\fp)},$$
we relate $|\lambda_\pi(\fp)|$ to the Hecke eigenvalues of $\sym^2(\pi)$, the automorphic representation on $\GL_3(\A_K)$. This causes us to need to break the proof into two cases: when $\pi$ is dihedral and when $\pi$ is non-dihedral.

\subsection{Galois representations}\label{subsec: dihedral}
First, we tackle the case when $\pi$ is a Galois representation, and hence has finite image. There is a Galois extension $L/K$ such that 
$$\pi:\Gal(L/K) \rightarrow \textrm{GL}_2(\C).$$ 
Additionally, the Hecke eigenvalues are of the form $$\lambda_\pi(\fp) = \Tr(\pi(\Frob_\fp)) = \chi_\pi(\Frob_\fp),$$
where $\chi_\pi$ denotes the character of the representation $\pi$. In this case, we can see that 
\begin{equation}
    S_\pi(X) = \sum_{\sigma\in \Gal(L/K)} |\chi_\pi(\sigma)| \sum_{\substack{N(\fp)\leq X \\ \Frob_\fp = \sigma}} \frac{1}{N(\fp)} - \log\log(X) + O_K(1).
\end{equation}
Recall that $\chi_\pi$ is invariant under conjugation and hence we can apply the Chebotarev density theorem to determine that 
\begin{equation}
    S_\pi(X) = \left(\frac{1}{|\Gal(L/K)|} \sum_{\sigma\in \Gal(L/K)} |\chi_\pi(\sigma)| - 1\right)\cdot  \log\log(X) + O_K(1).
\end{equation}

In order to establish Theorem \ref{thm: linear} in this case, we need to prove that
\begin{equation}\label{eq: galois rep < 1}
    \frac{1}{|\Gal(L/K)|} \sum_{\sigma\in \Gal(L/K)} |\chi_\pi(\sigma)| < 1.
\end{equation}
To do so, we apply Cauchy-Schwarz: 
\begin{equation}
    \frac{1}{|\Gal(L/K)|} \sum_{\sigma\in \Gal(L/K)}|\chi_\pi(\sigma)| \leq \left(\frac{1}{|\Gal(L/K)|} \sum_{\sigma\in \Gal(L/K)} |\chi_\pi(\sigma)|^2\right)^{1/2}.
\end{equation}
Since $\pi$ is an irreducible representation, we know that 
\begin{equation}
    \frac{1}{|\Gal(L/K)|} \sum_{\sigma\in \Gal(L/K)} |\chi_\pi (\sigma)|^2 =1.
\end{equation}
Observe that the Cauchy-Schwarz inequality is in fact a strict inequality unless $|\chi(\sigma)|$ is equal for every $\sigma\in \Gal(L/K)$. On the other hand, since $\pi$ is cuspidal, it is a nontrivial irreducible representation and hence $\chi(1) = \dim(\pi) \geq 2$. Thus, in this case, Cauchy-Schwarz is a strict inequality. 

Furthermore, since $|\chi_\pi(\sigma)| \in \Z$ for each $\sigma$, we can see that 
\begin{equation}S_\pi(X) \leq  - |\Gal(L/K)|^{-1}\cdot \log\log(X) + O_K(1),\end{equation}
and thus we have that 
\begin{equation}
    \sum_{N(\fa)\leq X} |\lambda_\pi(\fa)| \ll_K X \log(X)^{-|\Gal(L/K)|^{-1}}.
\end{equation}.

\begin{remark}
    For a fixed Galois representation, we can understand completely the $L^1$-norm of $\pi$, since we can compute all of the moments of $\pi$. Indeed, the asymptotic holds: $$\sum_{N(\fa)\leq X} |\lambda_\pi(\fa)| \sim c_\pi X\log(X)^{\beta_\pi}$$
    where $\beta_\pi$ is one fewer than the $L^1$-norm of the trace function of $\pi$.  
\end{remark}

\begin{remark}
    One can instead look at cuspidal representations of the Weil group $W_{K}$ with finite image and derive a similar logarithmic savings over the trivial bound for the $L^1$-norm of $\pi$. 
\end{remark}

\subsection{Using the second and fourth moment}\label{subsec: nondihedral}
Now, let us assume that $\pi$ is \textit{any} cuspidal automorphic representation of $\GL_2(\A_K)$ satisfying the Ramanujan-Petersson conjecture; we will use information about the $L^2$-norm and $L^4$-norm of $\pi$ to derive nontrivial savings on the $L^1$-norm of $\pi$. We follow the proof of Elliott, Moreno, and Shahidi \cite{EMS} regarding the Ramanujan $\tau$-function. Similar proofs also appear in the work of Chiriac and Yang \cite{ChiriacYang}, and Holowinsky \cite{HolowinskyShiftedConvolution}.

Let us define 
\begin{equation}\label{eq: def of delta}
    \delta := \inf_{-1\leq y\leq 2} y^{-2}(1+y/2-(1+y)^{1/2}) \approx 0.067.
\end{equation}
We remark that the definition of $\delta$ is strictly positive and independent of $\pi$.
By the definition of $\delta$, the following relation holds for any prime ideal $\fp$:
\begin{align}\label{eq: relation between absolute and symmetric powers}
    |\lambda_\pi(\fp)|-1&\leq \frac{1}{2}(\lambda_{\pi}(\fp)\lambda_{\overline{\pi}}(\fp)-1) - \delta \left(\lambda_{\pi}(\fp)\lambda_{\overline{\pi}}(\fp)-1\right)^2\\
    &\leq -\delta + \left(\frac{1}{2}+ 2\delta\right)\left(\lambda_{\pi}(\fp)\lambda_{\overline{\pi}}(\fp)-1\right) - \delta \left(\lambda_\pi(\fp)^2 \lambda_{\overline{\pi}}(\fp)^2 - 2\right).
\end{align}
(Observe that we use that the eigenvalues of the Hecke operators are real.)
Let $\pi\boxtimes \pi'$ denote the Rankin-Selberg convolution of $\pi$ and $\pi'.$
From the above inequality, we know 
\begin{equation}
    S_\pi(X) \leq -\delta \log\log(X) + \left(1/2 + 2\delta\right) \cdot T_{\pi\boxtimes \overline{\pi}}(X) - \delta \cdot T_{\pi\boxtimes \pi \boxtimes\overline{\pi}\boxtimes\overline{\pi}}(X) +O_K(1), 
\end{equation}
where we define 
\begin{equation}
    T_{\pi\boxtimes \overline{\pi}}(X)  = \sum_{N(\fp)\leq X} \frac{\lambda_{\pi\boxtimes\overline{\pi}}(\fp)-1}{N(\fp)}, \hspace{.5cm} T_{\pi\boxtimes \pi \boxtimes\overline{\pi}\boxtimes\overline{\pi}}(X) = \sum_{N(\fp)\leq X} \frac{\lambda_{\pi\boxtimes \pi \boxtimes \overline{\pi}\boxtimes\overline{\pi}}(\fp)-2}{N(\fp)}.
\end{equation}
In other words, $T_{\pi\boxtimes\overline{\pi}}$ measures the second moment of $\pi$ and $T_{\pi\boxtimes\pi\boxtimes\overline{\pi}\boxtimes\overline{\pi}}$ measures the fourth moment. At this point, we split into the case when $\pi$ is \textit{dihedral} and \textit{non-dihedral}. \\ 

If $\pi$ is a dihedral representation, then it must also be a representation of the Weil group $W_{K}$ (see \cite{KimShahidi} for a classification of representations with a noncuspidal symmetric power.) Moreover, since $\pi$ is dihedral, it will be a representation of $D_{2n}$ for some $n\in [3,\infty]$, where we include the infinite dihedral group and exclude $D_2$ and $D_4$ which have no two-dimensional irreducible representations.

Now if $\pi$ is dihedral Galois representation with $\Gal(L/K)\cong D_{2n}$, then \S\ref{subsec: dihedral} shows that $$\sum_{N(\fa)\leq X} |\lambda_\pi(\fa)|\ll X\log(X)^{-1/2n}.$$
However, this exponent is not uniform in the choice of cuspidal representation $\pi$. So, let us present a separate argument to give the uniform lower bound on $c_\pi$ in the case when $\pi$ is dihedral. 

If $\pi$ is a dihedral representation, then there must exist a quadratic extension $L/K$ such that $\pi = \Ind_{W_L}^{W_K} \theta$, where $\theta$ denotes a Hecke character. We can further assume that $\ord(\theta)>2$, as otherwise $\pi$ would be a two-dimensional irreducible representation of $D_4$ which do not exist. Since $\pi$ is an induced representation, when a prime ideal $\fp$ splits in $L$, say as $\fp = \frakP \overline{\frakP}$, we have that 
\begin{equation}\lambda_\pi(\fp) = \theta(\frakP) + \overline{\theta(\frakP)}.\end{equation}
Further, for all but finitely many primes $\fp$ such that $\fp$ does not split, $\lambda_\pi(\fp)=0$. Since $\pi$ is tempered at all finite places, for each prime $\fp$ that splits in $L$, we can assign an angle $\alpha_\fp \in [0,1]$ such that 
\begin{equation}
    \lambda_\pi(\fp) = |e^{2\pi i \alpha_\fp}+e^{-2\pi i \alpha_\fp}| = 2|\cos(2\pi \alpha_\fp)|.
\end{equation}
Now, we can write that 
\begin{equation}
    S_\pi(X) = \sum_{\substack{N(\fp)\leq X\\ \fp = \frakP\overline{\frakP}}} \frac{2|\cos(2\pi \alpha_\fp)|-1}{N(\fp)} -\frac{1}{2}\log\log (X)+  O_{\pi,K}(1).
\end{equation}

The cosine identity gives us the following useful fact:
 \begin{equation}
     2|\cos(2\pi \alpha_\fp)| \leq \frac{3}{2}+\frac{1}{2}\cos(4\pi \alpha_\fp).
 \end{equation}
 Thus,  
 \begin{equation}
     \sum_{\substack{N(\fp)\leq X \\ \fp = \frakP\overline{\frakP}}} \frac{2|\cos(2\pi \alpha_\fp)|-1}{N(\fp)} \leq  \sum_{\substack{N(\fp)\leq X \\ \fp = \frakP\overline{\frakP}}} \frac{1}{2N(\fp)} + \frac{\cos(2\pi \alpha_\fp)}{2N(\fp)} = \sum_{\substack{N(\fp)\leq X \\ \fp = \frakP\overline{\frakP}}} \frac{1}{2N(\fp)} + \frac{\theta(\frakP)^2 + \theta(\overline{\frakP})^2}{2N(\fp)}.
 \end{equation}
Evaluating the first of these summands, we can rearrange to see that 
\begin{equation}
    S_\pi(X) \leq \frac{1}{4}\log\log(X) - \frac{1}{2}\log\log(X) + \frac{1}{2}\sum_{N_{L/\Q}(\frakP)\leq X}\frac{\theta(\frakP)^2}{N_{L/\Q}(\frakP)}+O_{\pi,K}(1).
\end{equation}
Since $\ord(\theta)>2$, $\theta^2$ will be a nontrivial Hecke character of $L$. Hence, 
\begin{equation}
    \sum_{N_{L/\Q}(\frakP)\leq X} \frac{\theta(\frakP)^2}{N_{L/\Q}(\frakP)} = O_{\pi,K}(1).
\end{equation}
With this in mind, we can see that when $\pi$ is dihedral, 
\begin{equation}
    S_\pi(X) \leq  -\frac{1}{4}\log\log(X)+O_{\pi,K}(1),
\end{equation}
and in this case, we can take $c_\pi > 1/4.$ Take $\delta'=1/5$ for our future comparison, as the non-dihedral case will give the eventual lower bound on $c_\pi$ stated in Theorem \ref{thm: linear}. \\

So, it remains to analyze when $\pi$ is not dihedral. We claim that in this case:
\begin{equation}\label{eq: claim convergence of pi x pi}
     T_{\pi\boxtimes \overline{\pi}}(X) = O_\pi (1), \qquad T_{\pi\boxtimes \pi \boxtimes\overline{\pi}\boxtimes\overline{\pi}}(X) = O_\pi (1).
\end{equation}
Before we prove the above claim, we make the observation that both sums converge if and only if the following limits exist:
\begin{equation}\label{eq: limit pi x pi}
    \lim_{s\rightarrow 1^+} \zeta_K(s)^{-1} L(s,\pi\boxtimes \overline{\pi}),
\end{equation}
\begin{equation}\label{eq: limit pi x 4}
    \lim_{s\rightarrow 1^+} \zeta_K(s)^{-2} L(s,\pi\boxtimes\pi\boxtimes\overline{\pi}\boxtimes \overline{\pi}). 
\end{equation}
\medskip

Let us start with the $T_{\pi\boxtimes\overline{\pi}}(X)$ and the $L$-function $L(s, \pi\boxtimes \overline{\pi})$. Since $\pi$ is cuspidal, we know that $L(s,\pi\boxtimes\overline{\pi})$ has a \textit{simple pole} at $s=1$; consequently, the limit \eqref{eq: limit pi x pi} exists and hence the first part of \eqref{eq: claim convergence of pi x pi} is true. Another way of viewing this statement is that $L(s,\pi,\text{Ad})$ is analytic at $s=1$, where $\text{Ad}(\pi)$ is the Gelbart-Jacquet lift of \cite{GelbartJacquet}.

Next, we can decompose $\pi\boxtimes \overline{\pi} = 1 \boxplus \text{Ad}^2(\pi)$; here we use $\boxplus$ to denote the isobaric sum (see \cite{LanglandsIsobaric}). If $\pi$ is non-dihedral, we know that $\text{Ad}^2(\pi)$ is cuspidal. Now we consider $T_{\pi\boxtimes \pi\boxtimes\overline{\pi}\boxtimes\overline{\pi}}(X)$ and \eqref{eq: limit pi x 4}. From our decomposition of $\pi\boxtimes \overline{\pi}$, we know that $L(s,\pi\boxtimes\pi\boxtimes\overline{\pi}\boxtimes \overline{\pi})$ has the same analytic behavior at $s=1$ as the product $\zeta_K(s) L(s,\pi,\text{Ad}^2)^2 L( s,\text{Ad}^2(\pi)\boxtimes \text{Ad}^2(\overline{\pi}))$ -- we know that $L(s,\pi,\text{Ad}^2)$ is entire at $s=1$ since $\text{Ad}^2(\pi)$ is a cuspidal automorphic representation of $\GL_3(\A_K)$. Similarly, we know that since $\text{Ad}^2(\pi)$ is cuspidal, $L(s,\text{Ad}^2(\pi)\boxtimes \text{Ad}^2(\overline{\pi}))$ has a pole of order one at $s=1$; here $\text{Ad}^2(\pi)\boxtimes\text{Ad}^2(\overline{\pi})$ is a $\GL_3\times \GL_3$ Rankin-Selberg convolution. 
Hence, this $L$-function has a pole of \textit{order two} at $s=1$. Thus, the limit \eqref{eq: limit pi x 4} exists and this establishes the second half of claim \eqref{eq: claim convergence of pi x pi}. 

Hence, in the case when $\pi$ is non-dihedral, we know that 
\begin{equation}
    \sum_{N(\fa)\leq X} |\lambda_\pi(\fa)| \ll_\pi X \log(X)^{-\delta}.
\end{equation}
Now, taking $c= \min(\delta, \delta')$, this completes the cuspidal case of Theorem \ref{thm: linear}.

\subsection{Lower bounds when $\pi$ is not cuspidal}\label{subsec: linear lower}
Let $\pi$ be a non-cuspidal representation of $\GL_2(\A_K)$, so $\pi$ has the isobaric decomposition (see \cite{LanglandsIsobaric}) $$\pi = \psi_1 \boxplus \psi_2,$$
where $\psi_1,\psi_2$ are $\GL_1(\A_K)$-representations, i.e. quasi-characters. 
Since $\pi$ is tempered at each finite place, we take $\psi_1$ and $\psi_2$ to be unitary. 

Let us assume that $|\lambda_\pi(\fp)|$ satisfies the assumptions of Theorem \ref{thm: lower bound sieve} for now (we will check this at the end). This would then imply that 
\begin{equation}
    \sum_{N(\fa)\leq X}|\lambda_\pi(\fa)| \gg_K X \exp\left(\sum_{N(\fp)\leq X} \frac{|\lambda_\pi(\fp)|-1}{N(\fp)}\right).
\end{equation}
Now, since $\pi = \psi_1\boxplus \psi_2$, we know that 
\begin{equation}
    \lambda_\pi(\fp) = \psi_1(\fp) + \psi_2(\fp).
\end{equation}
Hence, we can estimate:
\begin{equation}
    S_\pi(X) = \sum_{N(\fp)\leq X} \frac{|\psi_1(\fp)+\psi_2(\fp)|-1}{N(\fp)} = \sum_{N(\fp)\leq X} \frac{|1+\psi_1^{-1}\psi_2(\fp)|-1}{N(\fp)}.
\end{equation}
Define the character $\psi:= \psi_1^{-1}\psi_2$ and define for each prime $\fp$, the angle $\theta_\fp\in [0,1]$ such that $\psi(\fp) = e(2\pi i \theta_\fp).$ We have that 
\begin{equation}\label{eq: S_pi in terms of cos}
    S_\pi(X) = \sum_{N(\fp)\leq X} \frac{\sqrt{2+2\cos(2\pi \theta_\fp)}-1}{N(\fp)}.
\end{equation}

We split our analysis into two cases: when $\psi=1$ and when $\psi$ is nontrivial. First, if $\psi=1$, then $\theta_\fp = 0$ for all primes $\fp$. So, we have that 
\begin{equation}
    S_\pi(X) = \sum_{N(\fp)\leq X} \frac{1}{N(\fp)} = \log\log(X)+O_K(1),
\end{equation}
and thus 
\begin{equation}
    \sum_{N(\fa)\leq X} |\lambda_\pi(\fa)| \gg_K X\log(X).
\end{equation}
This completes the lower bound in this case.

Now, if $\psi\neq 1$, then we know that for any $A>0$: 
\begin{equation}\label{eq: equidistribute theta}
    \sum_{N(\fp)\leq X} \cos(2\pi \theta_\fp) = O_\psi(X\log(X)^{-A}).
\end{equation}
Consequently, the following hold:
\begin{equation}
    \sum_{N(\fp)\leq X} 2+2\cos(2\pi \theta_p) = \frac{2X}{\log(X)} + O_\psi(X(\log X)^{-A}).
\end{equation}
Observing that $\max_{\theta\in [0,1]}\sqrt{2+2\cos(2\pi \theta_p)} = 2$, we establish that
\begin{equation}\label{eq: sqrt(2+2cos)}
    \sum_{N(\fp)\leq X} \sqrt{2+2\cos(2\pi \theta_p)} \geq \frac{X}{\log(X)} + O_\psi(X(\log X)^{-A}).
\end{equation}
Hence, through partial summation, we have that 
\begin{equation}\label{eq: distribution of angles}
   \sum_{N(\fp)\leq X} \frac{\sqrt{2+2\cos(2\pi \theta_\fp)}}{N(\fp)} \geq \log\log(X) + O_{\pi,K}(1).
\end{equation}
Consequently, from the equation above in combination with \eqref{eq: S_pi in terms of cos}, we know 
\begin{equation}
    S_\pi(X) \ll_{\pi,K} 1.
\end{equation}
As such, there can not exist any positive constant $c_{\pi}>0$ such that $S_\pi(X) \leq -c_\pi \log\log(X) + o_{\pi,K}(\log\log(X))$, as would be necessary for the upper bound in Theorem \ref{thm: linear} to exist. 
Thus, in order to complete the proof of Theorem \ref{thm: linear}, it suffices to prove that the multiplicative function $|\lambda_\pi(\fp)|$ satisfies the assumptions of Theorem \ref{thm: lower bound sieve}.
\medskip

Finally, let us check that when $\pi = \psi_1\boxplus \psi_2$
\begin{equation}
    \sum_{N(\fp)\leq y} |\lambda_\pi(\fp)|\log(N(\fp)) \gg y.
\end{equation}
Again, let $\psi = \psi_1^{-1}\psi_2$ and $\psi(\fp) = e(2\pi i \theta_\fp).$ Then we can write 
\begin{equation}
    \sum_{N(\fp)\leq y} |\lambda_\pi(\fp)|  \log(N(\fp)) = \sum_{N(\fp)\leq y} |1+\psi(\fp)| \log(N(\fp)) = \sum_{N(\fp)\leq y} \sqrt{2+2\cos(2\pi \theta_\fp)}\cdot \log(N(\fp)).
\end{equation}
By partial summation in combination with \eqref{eq: sqrt(2+2cos)}, we see that:
\begin{equation}
    \sum_{N(\fp)\leq y} \sqrt{2+2\cos(2\pi \theta_\fp)}\cdot \log(N(\fp))\gg y
\end{equation}
as desired. Thus $|\lambda_\pi(\fp)|$ satisfies the conditions of Theorem \ref{thm: lower bound sieve}. \qed

\section{Higher degree polynomials}\label{sec: poly}

In the following section, we prove Theorem \ref{thm: poly} and Theorem \ref{thm: noncuspidal}. Before doing so, we will review base change for $\GL_1$ and $\GL_2$ automorphic representations, as this theory will be critical in both sections. For a more comprehensive discussion of base change, see \cite{GerardinLabesseBaseChange,LanglandsBaseChange}.

Let us recall the following notation:
\begin{itemize}
    \item $F$ is a fixed number field; in our applications it will always be $F=\Q$,
    \item $G=\GL_1$ or $\GL_2$ defined over $F$, 
    \item For a number field extension $E/F$, $G_{E/F}$ is the restriction of scalars of $G$ over $E$, i.e. $G_{E/F}(\A_F) = G(\A_E)$,
    
    \item The Weil group of a number field $F$ is given by a tuple $(W_F, \varphi, \{r_E\}_{E/F})$ such that $W_F$ is a topological group with a continuous homomorphism:
    $$\varphi: W_F \rightarrow \Gal(\overline{F}/F)$$
    with dense image that satisfies:
    \begin{itemize}
    \item For any finite extension $E/F$ with $E\subset \overline{F}$, define $W_E = \varphi^{-1}(\Gal(\overline{F}/E))$. Since $\varphi$ is continuous, $W_E$ is open in $W_F$; additionally, it induces a bijection 
    $$W_F/W_E \xrightarrow{\sim} \Gal(\overline{F}/F) / \Gal(\overline{F}/E) \cong \Gal(E/F),$$
    where the last isomorphism holds when $E/F$ is Galois. 
    \item For each finite extension $E/F$, the map $r_E$ gives an isomorphism of topological groups 
    $$r_E: \A_E^\times / E^\times \xrightarrow{\sim} W_E^{\text{ab}},$$
    where $W_E^{\text{ab}}$ is the maximal abelian Hausdorff quotient of $W_E$.
    Furthermore, these maps $r_E$ satisfy four conditions that are described explicitly in \cite[(W1),(W2),(W3),(W4)]{TateNTBackground}.
    \end{itemize}
    \item For a one-dimensional representation $\theta$ of $W_E$, $\Ind_{W_E}^{W_F}\theta$ denotes automorphic induction to the extension $E/F$.
\end{itemize}

We also set $\Pi(G)$ to denote the set of irreducible admissible automorphic representations of $G(\A_F)$. The base change map sends: 
$$\Pi(G)\rightarrow \Pi(G_{E/F})$$
$$\pi \mapsto \pi_{E/F}.$$
Let us assume that $E/F$ is a solvable number field extension; hence there exists a chain of subfields $$F=E_0\subsetneq E_1\subsetneq E_2\subsetneq \hdots \subsetneq E_{n-1} \subsetneq E_n = E$$
    where $E_{i+1}/E_i$ is a cylic extension of prime order. Base change on $L$-groups gives us a map $$\Pi(G) \rightarrow \Pi(G_{E_1/F}) \rightarrow \hdots \rightarrow \Pi(G_{E_{n-1}/F})\rightarrow \Pi(G_{E/F})$$
that is compatible with local data; when $G=\GL_2$, this map was constructed and proved by Langlands in \cite{LanglandsBaseChange}. \medskip
\subsubsection{The $\GL_1$ base change problem}
First, we review the case of $\GL_1$-base change. Let $\theta$ be a one-dimensional representation of $W_{F}$, i.e. a quasicharacter of $\A_{F}^\times / F^\times$. Let $E/F$ denote a cyclic extension. The base change lift $$\Pi(\GL_1(\A_F)) \rightarrow \Pi(\GL_1(\A_{E}))$$
is given explicitly by the map $$\theta \mapsto \theta\circ N_{E/F}=:\theta_{E}.$$
In particular, for an ideal $I\subset E$, $$\lambda_{\theta_{E}}(I) = \theta(N_{E/F}(I)) = \lambda_{\theta_{F}}(N_{E/F}(I)).$$
Next, observe that $\widehat{\Gal(E/F)}$ are themselves quasi-characters that are trivial on $N_{E/F}(\A_{E}^\times/ E^\times)$ and hence will be trivial after base change to $E.$ 
\begin{prop}[{\cite[Proposition 1]{GerardinLabesseBaseChange}}]\label{prop: GL1base change}
Let $E/F$ be a cyclic extension. 
    The lifting $\theta_{E}$ will be cuspidal if and only if there exists $\zeta\in \widehat{\Gal(E/F)}$ such that $\theta = \zeta.$ 
\end{prop}

\subsubsection{The $\GL_2$ base change problem}
Now, we discuss the more complicated case of $\GL_2$-base change. The base change for global fields is discussed in \cite[Theorem 2]{GerardinLabesseBaseChange}. We recall the relevant parts of the theorem:
\begin{thm}[{\cite[Theorem 2 (a),(b)]{GerardinLabesseBaseChange}}]\label{thm: base change GL2}
Let $\pi\in \Pi(G)$ be a cuspidal representation. Then there exists a unique lifting $\pi_{E/F}\in \Pi(G_{E/F})$. Since $E/F$ is solvable, denote the chain of cyclic extension: 
$$F=E_0\subsetneq E_1 \subsetneq \hdots \subsetneq E_{n-1}\subsetneq E_n=E.$$
Then, $\pi_{E/F}$ is not cuspidal if and only if there exists $E_{i+1}/E_{i}$ a quadratic extension and $\theta$ a quasicharacter of $\A_{E_{i+1}}^\times/ E_{i+1}^\times$ such that $$\pi_{E_i} = \Ind_{W_{E_{i+1}}}^{W_{E_{i}}} \theta.$$ 
In this case, $\pi_{E_{i+1}}$ is a principal series given by quasicharacter $\theta$, i.e. $\pi(\theta,^{\sigma}\theta).$
\end{thm}
We further make the observation that for $\pi$ an automorphic representation of $\GL_2(\A_F)$, the Fourier coefficients of $\pi_{E/F}$ at prime ideals $\frakP\subset \calO_E$ satisfying that $N_{E/\Q}(\frakP)$ is prime will be given by the relation
\begin{equation}\label{eq: relation between coeff of base change}
    \lambda_{\pi_{E/F}}(\frakP) = \lambda_{\pi}(N_{E/F}(\frakP)).
\end{equation}
This connection between the Hecke eigenvalues of $\pi$ and $\pi_{E/F}$ will be the key relationship between the correlation sum of Theorem \ref{thm: poly} and this automorphic picture. 

\subsection{Cuspidal base change}
In this subsection, we prove Theorem \ref{thm: poly}, i.e. when the base change of $\pi$ to the splitting field $K$ of $P(x)$ is cuspidal. 

\begin{proof}[Proof of Theorem \ref{thm: poly}]
We will reduce Theorem \ref{thm: poly} to the estimates in Theorem \ref{thm: linear}, where $K$ is the splitting field of $P(x)$. Again, let us start by applying the sieve (Theorem \ref{thm: sieve nair}): 
\begin{equation}
    \sum_{n\leq X} |\lambda_\pi(|P(n)|)| \ll_{P} X \exp\left(\sum_{p\leq X} \frac{\varrho(p)(|\lambda_\pi(p)|-1)}{p}\right).
\end{equation}
We observe that since $\varrho(p) = \#\{x\bmod p: P(x) = 0\}$, at all but finitely many primes, 
$$\varrho(p) = \#\{\fp\subset \calO_K: N_{K/\Q}(\fp) = p\}.$$
Since $P(x)$ is irreducible, we have 
\begin{equation}
    \sum_{p\leq X}\frac{\varrho(p)|\lambda_\pi(p)|}{p} - \frac{\varrho(p)}{p} = \sum_{N(\fp)\leq X} \frac{|\lambda_\pi(N(\fp))|-1}{N(\fp)} + O_K(1). 
\end{equation}
Now by \eqref{eq: relation between coeff of base change}, we know that this gives the relationship
\begin{equation}
    \sum_{p\leq X}\frac{\varrho(p)(|\lambda_\pi(p)|-1)}{p} = \sum_{N(\fp)\leq X} \frac{|\lambda_{\pi_K}(\fp)|-1}{N(\fp)} + O_K(1) = S_{\pi_K}(X) + O_K(1),
\end{equation}
using the notation of \S\ref{sec: linear}. Since we have assumed that $\pi_K$ is a cuspidal automorphic representation of $\GL_2(\A_K)$, by the proof of Theorem \ref{thm: linear}, for some $c\geq 0.066,$
\begin{equation}
    S_{\pi_K}(X) \leq  -c \log\log(X)+O_{\pi,K}(1).
\end{equation}
Thus, we establish the desired estimate 
\begin{equation}
    \sum_{n\leq X} |\lambda_\pi(|P(n)|)| \ll_{\pi,P} X\log(X)^{-0.066}.
\end{equation}
\end{proof}

\subsection{Non-cuspidal base change}\label{subsec: noncuspidal}
Let $K$ again denote the splitting field of $P(x)$ and assume that it is solvable. Additionally, assume that $\pi_K$ is not cuspidal and want to determine the size of $$\sum_{n\leq X} |\lambda_\pi(|P(n)|)|.$$

\begin{proof}[Proof of Theorem \ref{thm: noncuspidal}]
First, let us determine some structure on these $\pi$ such that $\pi_K$ is not cuspidal. 
Since $K$ is solvable, there exists a chain of subfields: 
\begin{equation}
    \Q =E_0 \subsetneq E_1 \subsetneq \hdots \subsetneq E_{n-1}\subsetneq E_n = K
\end{equation}
with $E_{i+1}/E_i$ cyclic extensions of prime order. Since $\pi_K$ is not cuspidal, by Theorem \ref{thm: base change GL2}, we know that there exists some index $i$ such that $[E_{i+1}:E_i] = 2$ and $$\pi_{E_i} = \Ind_{W_{E_{i+1}}}^{W_{E_i}}\theta,$$
where $\theta$ is a quasi-character of $\A_{E_{i+1}}^\times/E_{i+1}^\times.$ Furthermore, this fact tells us that $$\pi = \Ind_{W_{E_{i+1}}}^{W_\Q} \theta.$$
Since $\pi$ is a two-dimensional representation, we claim that we must have that $E_i = \Q$ and $E_{i+1}$ a quadratic extension. Assume not -- then $[E_{i+1}:\Q]>2$. Since $\theta$ denotes a one-dimensional representation, $\dim(\pi) = [E_{i+1}:\Q]>2$, immediately giving us a contradiction to our assumption that $\pi$ is a two-dimensional representation. Finally, since $E_{i+1}$ must be a quadratic extension of $\Q$ and $E_i\subsetneq E_{i+1}$, we must have that $E_{i} = \Q.$

Let $\fp$ be a prime ideal of $\calO_{F}$ sitting over $p$. Then we know that 
\begin{equation}
    \lambda_{\pi}(p) = \sum_{\substack{N_{F/\Q}(\fp) = p}}\theta(\fp).
\end{equation}
Furthermore, for $\frakP\subset \calO_K$ be a prime ideal sitting above $\fp$, we also can see that 
\begin{equation}
    \lambda_{\pi_K}(\frakP) = \lambda_{\pi}(N_{K/\Q}(\frakP)) = \sum_{\substack{\fp \subset \calO_{F} \\ N_{K/\Q}(\frakP) = N_{F/\Q}(\fp)}} \theta(\fp).
\end{equation}

We return to the focus of Theorem \ref{thm: noncuspidal}: 
\begin{equation}
    \sum_{n\leq X} |\lambda_\pi(P(n))| \ll_P X \exp\left(\sum_{p\leq X} \frac{\varrho(p)(|\lambda_\pi(p)|-1)}{p}\right).
\end{equation}
Here we have again applied Theorem \ref{thm: sieve nair}. Since we know that 
\begin{equation}
    \sum_{p\leq X} \frac{\varrho(p)(|\lambda_\pi(p)|-1)}{p} = \sum_{\substack{\frakP\subset \calO_K \\ N_{K/\Q}(\frakP) \leq X}} \frac{|\lambda_{\pi_K}(\frakP)|-1}{N_{K/\Q}(\frakP)} +O_{K,\pi}(1),
\end{equation}
we find that we want to determine the size of 
\begin{align}
    S_{\pi_K}(X) &:= \sum_{\substack{\frakP\subset \calO_K \\ N_{K/\Q}(\frakP) \leq X}} \frac{|\lambda_{\pi_K}(\frakP)|-1}{N_{K/\Q}(\frakP)} = \sum_{\substack{\frakP\subset \calO_K \\ N_{K/\Q}(\frakP) \leq X}} \frac{1}{N_{K/\Q}(\frakP)} \Bigg(\Big|\sum_{\substack{\fp\subset \calO_{F} \\ N_{K/\Q}(\frakP) = N_{F/\Q}(\fp)}} \hspace{-.5cm}\theta(\fp)\Big|-1\Bigg) + O_{K,\pi}(1).
\end{align}
Since $\lambda_{\pi_K}(\frakP)$ only depends on the value of $N_{K/\Q} (\frakP) = p$, we group the summand by the prime ideals sitting above $p$. It is worth observing as well that since we sum over prime ideals, only those primes satisfying that $N_{K/\Q}(\frakP)$ is prime will contribute to the growth of the summand; the other primes can be grouped into the $O_{K,\pi}(1)$ term. Thus, we achieve that 
\begin{equation}
   S_{\pi_K}(X) = \sum_{p\leq X} \frac{\#\{\frakP \subset \calO_K: N_{K/\Q}(\frakP)=p\}}{p} \left(\left|\sum_{\substack{\fp\subset \calO_{F} \\ N_{F/\Q}(\fp) = p }}\theta(\fp)\right|-1\right) + O_K(1).
\end{equation}

Since $\theta$ is a unitary character of $F$ and $\lambda_{\pi}(p)\in \R$, we have a stronger replacement for \eqref{eq: relation between absolute and symmetric powers} that only utilizes the second moment (rather than both the second and fourth moment): 
\begin{equation}
    \left|\sum_{N_{F/\Q}(\fp) = p} \theta(\fp)\right| \leq \frac{3}{2} + \frac{1}{4}\left(\sum_{N_{F/\Q}(\fp) = p}\theta^2(\fp)\right).
\end{equation}
Furthermore, if $p$ is a prime such that there does not exists $\fp\subset \calO_{F}$ with $N_{F/\Q}(\fp) = p$, then $\lambda_{\pi}(p) = 0$. 
Thus, we have the upper bound 
\begin{multline*}
    S_{\pi_{K}}(X) \leq \frac{1}{2}\sum_{\substack{p\leq X \\ \exists \fp: N_{F/\Q}(\fp)=p}} \frac{\#\{\frakP\subset \calO_K: N_{K/\Q}(\frakP)=p\}}{p} \\
     - \sum_{\substack{p\leq X \\ \not\exists \fp: N_{F/\Q}(\fp)=p}} \frac{\#\{\frakP\subset \calO_K: N_{K/\Q}(\frakP)=p\}}{p} \\
     + \frac{1}{4}\sum_{\substack{p\leq X \\ \exists \fp: N_{F/\Q}(\fp)=p}} \frac{\#\{\frakP\subset \calO_K: N_{K/\Q}(\frakP)=p\}}{p} \times \left(\sum_{\substack{\fp\subset \calO_{F}\\ N_{F/\Q}(\fp) = p}} \theta^2(\fp)\right).
\end{multline*}
We make the following comments about these sums: (1) for the first sum, this is equivalent to counting
\begin{equation}
    \frac{1}{2}\sum_{N_{F/\Q}(\fp)\leq X}\frac{\#\{\frakP\subset \calO_K: N_{K/F}(\frakP) = \fp\}}{N_{F/\Q}(\fp)} + O_K(1) = \frac{1}{2}\log\log(X) + O_K(1),
\end{equation}
(2) the second sum is empty since $F\subset K$, (3) after counting the number of ideals with $N_{F/\Q}(\fp) = p$, the third sum becomes 
\begin{multline}
    \frac{1}{2} \sum_{\substack{N_{F/\Q}(\fp)\leq X}} \frac{\#\{\frakP\subset \calO_K: N_{K/F}(\frakP)= \fp\}}{N_{F/\Q}(\fp)}  \cdot \theta^2(\fp) + O_K(1)
    = \frac{1}{2} \sum_{N_{K/\Q}(\frakP)\leq X} \frac{\theta^2(N_{K/F}(\frakP))}{N_{K/\Q}(\frakP)} + O_K(1).
\end{multline}
Thus, we know that 
\begin{equation}
    S_{\pi_K}(X) \leq \frac{1}{2}\log\log(X) + \frac{1}{2} \sum_{N_{K/\Q}(\frakP)\leq X} \frac{\theta^2(N_{K/F}(\frakP))}{N_{K/\Q}(\frakP)} + O_K(1),
\end{equation}
the upper bound of Theorem \ref{thm: noncuspidal} (1) holds so long as $\theta^2\circ N_{K/F}$ is not a trivial character; in other words, we need to understand the \textbf{$\GL_1$-base change} of $\theta^2$ to the extension $K/E_i$. Observe that if $\theta$ has infinite order, we never have that $\theta^2\circ N_{K/F}$ is trivial and Theorem \ref{thm: noncuspidal} (1) falls. 

\subsubsection{Return to the $\GL_1$-base change problem}

By Proposition \ref{prop: GL1base change}, we know that $\theta^2_K$ is trivial, i.e. noncuspidal, if and only if $\theta^2 \in \widehat{G}$ for $G=\Gal(K/F).$ Let us denote $\psi = \theta^2$. We want to find another interpretation of our condition for when $\psi_K\equiv 1$. Consider the following diagram, where $L_\psi$ (resp. $L_\theta$) denotes the ray class field of $\psi$ (resp. $\theta$). Since $\ker(\theta) \subset \ker(\psi)$, we can take $L_\psi\subset L_\theta$. 

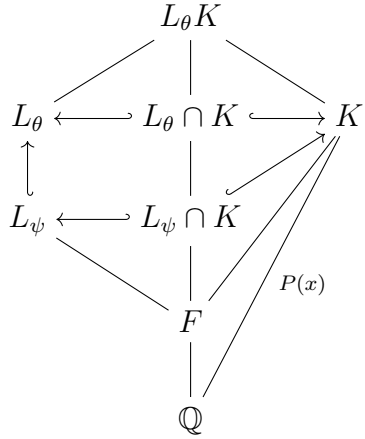
\[\begin{tikzcd}[cramped]
	& {L_\theta K} \\
	{L_\theta} & {L_\theta\cap K} & K \\
	{L_\psi} & {L_\psi\cap K} \\
	& F \\
	& {\mathbb{Q}}
	\arrow[no head, from=1-2, to=2-1]
	\arrow[no head, from=1-2, to=2-2]
	\arrow[no head, from=1-2, to=2-3]
	\arrow[hook', from=2-2, to=2-1]
	\arrow[hook, from=2-2, to=2-3]
	\arrow[no head, from=2-2, to=3-2]
	\arrow[no head, from=2-3, to=4-2]
	\arrow["{P(x)}", no head, from=2-3, to=5-2]
	\arrow[hook', from=3-1, to=2-1]
	\arrow[no head, from=3-1, to=4-2]
	\arrow[hook, from=3-2, to=2-3]
	\arrow[hook', from=3-2, to=3-1]
	\arrow[no head, from=3-2, to=4-2]
	\arrow[no head, from=4-2, to=5-2]
\end{tikzcd}\]

Let $G_\psi$ denote the Galois group of $L_\psi\cap F$. 
We know that $\psi_K=1$ if and only if $\psi \in \widehat{G}$; we claim this can occur if and only if $\psi\in \widehat{G_\psi}$, i.e. if and only if $\psi_{L_\psi\cap K}=1$. This is clear since $\psi$ is a character of $\Gal(L_\psi/F)$, and hence the interpretation of $\psi$ as an element of $\widehat{G}$ must be as a character on $G_\psi$ extended trivially to $G$. 
Now by definition of the ray class field of $L_\psi$, this in fact tells us that $L_\psi = L_\psi\cap K$, and hence we have that $L_\psi\subset K$. 

\begin{lemma}\label{lem: base change trivial}
    Let $\psi$ be a quasi-character of $F$. Then $\psi_K=1$ if and only if $L_\psi\subset K$, where $L_\psi$ denotes the kernel field of $\psi$. 
\end{lemma}

Now we recall that in our set up $\psi = \theta^2$; we want to determine when $L_{\theta^2}\subset K$ in order to determine when $\theta^2_K=1$. Instead of analyzing the extension $L_\psi$, we can instead work with $L_\theta$. We can view $\theta^2$ as a character of the Galois group $\Gal(L_\theta/F)$ and see that $\theta^2_K=1$ if and only if $\theta^2_{L_\theta\cap K} =1$, and thus, if and only if $\theta^2 \in \widehat{G_\theta}$ (where $G_\theta = \Gal(L_\theta\cap K/F)$). In turn, this occurs if and only if $\Gal(L_\theta/L_\theta\cap K) \subset \ker(\theta^2),$ where $\ker(\theta^2)$ is considered as a subgroup of $\Gal(L_\theta/F)$. This completes Theorem \ref{thm: noncuspidal} (1).
\begin{remark}
    We note that $\ker(\theta)= \{1\}\subset \Gal(L_\theta/F)$, but we can have that $\{1\}\subsetneq \ker(\theta^2)$. So, this condition does not necessarily imply that $L_\theta\subset K$. 
\end{remark}

It remains to prove the second part of the theorem when $\Gal(L_\theta/L_\theta\cap K) \subset \ker(\theta^2).$ We return to our original sum: 
\begin{align}
    \sum_{n\leq X} |\lambda_\pi(|P(n)|)| &\geq \sum_{\substack{n\leq X \\ P(n) \text{ squarefree}}} |\lambda_\pi(|P(n)|)|\\
    &= \sum_{\substack{n\leq X \\ P(n) \text{ squarefree}}} \left|\sum_{\substack{\fa\subset \calO_F \\ N_{F/\Q}(\fa) = |P(n)|}}\theta(\fa)\right| \\
    &= \sum_{\substack{n\leq X \\ P(n) \text{ squarefree}}} \prod_{\substack{p\mid P(n)\\ \exists \fp\subset \calO_F: N_{F/\Q}(\fp) = p}} |1+\theta^2(\fp)|.
\end{align}
Above, $\fp$ denotes a choice of prime ideal of $\calO_F$ such that $N_{F/\Q}(\fp) = p$. Since $p\mid P(n)$, there exists a $\frakP\subset \calO_K$ such that $N_{K/\Q}(\frakP) = p$ and that $N_{K/F}(\frakP) = \fp$; let $\frakP' = N_{K/L_\theta\cap K}(\frakP)$. Since $\Gal(L_\theta/L_\theta\cap K) \subset \ker(\theta^2)$ and hence $\theta_{L_\theta\cap K}^2 = 1$, we know that 
$$\theta^2(\fp) = \theta^2(N_{L_\theta\cap K/F}(\frakP')) = \theta^2_{L_\theta\cap K}(\frakP') = 1.$$ Consequently, 
\begin{equation}
    \sum_{n\leq X} |\lambda_\pi(|P(n)|)| \geq \sum_{\substack{n\leq X \\ P(n) \text{ squarefree}}} \prod_{\substack{p\mid P(n) \\ \exists \fp: N_{F/\Q}(\fp) = p}} 2 = \sum_{\substack{n\leq X \\ P(n)\text{ squarefree}}} \sum_{\substack{\fa\subset \calO_F\\ N_{F/\Q}(\fa) = |P(n)|}}1.
\end{equation}
This completes the proof of Theorem \ref{thm: noncuspidal} (2).
\end{proof}

\bibliographystyle{abbrv}
\bibliography{biblio}

\end{document}